\documentclass{article}
\usepackage{graphicx} 
\usepackage{cite}
\usepackage{amssymb,amsthm,amsmath}
\title{On invex functions with same $\eta$ in single- and multivalued nonsmooth optimization with Clarke's subdifferential}
\author{Ville Rinne, Yury Nikulin\footnote{Corresponding author}, Marko M. Mäkelä}
\date{February 2025}
\theoremstyle{definition}
\newtheorem{maar}{Definition}[section]
\newtheorem{lemma}[maar]{Lemma}
\newtheorem{theorem}[maar]{Theorem}
\newtheorem{example}[maar]{Example}

\newtheorem{rem}[maar]{Remark}
\newtheorem{assum}[maar]{Assumption}

\begin{document}
\maketitle
\begin{abstract}
In this paper, a finite family of nonsmooth locally Lipschitz continuous functions that are invex with respect to the same function $\eta$ are characterized in terms of their scalarized counterparts.
\end{abstract}


\newpage
\section{Introduction}
Invexity is an important concept in mathematical optimization, extending the ideas of convexity to a wider range of problems. In traditional convex optimization, a function is convex if any line segment between two points on the function's graph lies above or on the graph itself. This property ensures that any local minimum is also a global one, making convex optimization problems relatively easier to solve.

However, not all real-world problems exhibit the strict property required for convexity. This is where invexity becomes useful. A function is considered invex if a differentiable function (usually denoted $\eta$) can transform the original problem into one in which sufficient conditions still apply for Karush-Kuhn-Tucker (KKT) conditions. These KKT conditions are crucial for determining the optimality in constrained optimization. Unlike convexity, invexity allows for more flexibility, making it possible to solve problems where the objective function might not be convex, but still retains some structural properties guaranteeing global optimality.

The concept of invexity broadens the scope of optimization by enabling the solution of nonconvex problems under more relaxed conditions, making it highly valuable in fields such as economics, engineering, and operations research \cite{CravenApp}. In essence, invexity provides a powerful tool for solving a wider array of optimization problems without sacrificing the assurance of finding optimal solutions.

In this paper, we follow the idea of \cite{Leg} and analyze invexity with respect to the same $\eta$ function for a finite set of nonsmooth locally Lipschitz continuous (LLC) functions in single- and multivalued nonsmooth optimization. Invexity with respect to the same function $\eta$, is a desirable theoretical property for a finite set of LLC functions unifying them. It can efficiently be used in optimality conditions, algorithm designing, and so on. The paper is organized as follows. Section \ref{sec2} contains all the necessary preliminaries. Section \ref{sec3} discusses invex nonsmooth functions for single and multivalued functions. Two theorems providing a link between optimality and invexity are highlighted. Section \ref{sec4} contains the main result of the paper. We start by presenting the known result of \cite{Leg} for differentiable cases. Then we give a counterexample showing that a straightforward generalization of the result is not possible for the case of nonsmooth LLC functions without Clarke’s regularity assumption. Then we formulate the main result of the paper (Theorem \ref{MTh}) showing the equivalence of four certain conditions characterizing V-invex functions with respect to the same $\eta$.

\section{Preliminaries}\label{sec2}
In this section, we discuss functions that are neither continuously differentiable nor convex. They have points in their domain where their gradient is not continuous. In these points, we define Clarke's generalized directional derivative and subgradient. 

We start by defining classical concepts from nonsmooth analysis - local Lipschitz continuity (LLC) and Lipschitz continuity (LC).

\begin{maar} {\cite{BaKaMa}} A function $f:\mathbb R^n \rightarrow \mathbb R$ is \emph{LLC at a point} $\boldsymbol x \in \mathbb  R^n$ if there exist scalars $K_{\boldsymbol{x}} > 0$ and $\delta_{\boldsymbol{x}} > 0$ such that
$$|f(\boldsymbol y) - f(\boldsymbol z)| \leq K_{\boldsymbol{x}} \left\|\boldsymbol y - \boldsymbol z\right\|,\ \forall\ \boldsymbol y, \boldsymbol z \in B(\boldsymbol x;\delta_{\boldsymbol{x}}),$$
\noindent
where $B(\boldsymbol x;\delta_{\boldsymbol{x}})$ denotes the open ball with the center $\boldsymbol{x}$ and radius $\delta_{\boldsymbol{x}}$. A function $f:\mathbb R^n \rightarrow \mathbb R$ is \emph{LLC} if it is LLC at every point belonging to $\mathbb R^n$. 
\end{maar}

\begin{maar} {\cite{BaKaMa}}
A function $f:\mathbb R^n \rightarrow \mathbb R$ is \emph{LC} if there exists a scalar $K$ such that
$$|f(\boldsymbol y) - f(\boldsymbol z)| \leq K \left\|\boldsymbol y - \boldsymbol z\right\|,\ \forall\ \boldsymbol y, \boldsymbol z \in \mathbb R^n.$$
\end{maar}

\begin{maar}
A vector-valued function $\boldsymbol F:\mathbb R^n \rightarrow \mathbb R^p$ is LLC if each vector component is LLC.  
\end{maar}

We continue by defining the Clarke generalized directional  (sometimes referred to as the limiting directional) derivative for LLC functions and discuss its properties.
\begin{maar} {\cite{Cla}}\label{GDD} Let $f:\mathbb R^n \rightarrow \mathbb R$ be LLC at $\boldsymbol x^* \in \mathbb R^n$. The \emph{Clarke's generalized directional derivative} of $f$ at $\boldsymbol x^*$ in direction $\boldsymbol d \in \mathbb R^n$ is defined as follows:
$$f^o(\boldsymbol x^*; \boldsymbol d) = \limsup_{\boldsymbol y \rightarrow \boldsymbol x^*,\ t \rightarrow 0_+}\frac{f(\boldsymbol y + t \boldsymbol d) - f(\boldsymbol y)}{t},$$ 
or equivalently,
$$f^o(\boldsymbol x^*; \boldsymbol d) = \inf_{\delta > 0}\ \sup_{\substack{\Vert y-x^*\Vert \leq \delta,\\ 0<t<\delta}}\frac{f(\boldsymbol y + t \boldsymbol d) - f(\boldsymbol y)}{t}.$$
\end{maar}

Note that for LLC functions, we always have  $f^o(\boldsymbol x^*; \boldsymbol d) < \infty$. In the definition of the classical directional derivative, the base point for taking differences is a fixed vector $\boldsymbol x^*$. In the generalized directional derivative, they are taken from a variable vector $\boldsymbol y$, which approaches $\boldsymbol x^*$.

Next, we define Clarke's generalized subdifferential (and subgradient as its element) using Clarke's generalized directional derivative and discuss its properties.

\begin{maar}{\cite{Cla}}\label{subdif} Let $f:\mathbb R^n \rightarrow \mathbb R$ be LLC function at a point $\boldsymbol x^* \in \mathbb R^n$. \emph{Clarke's subdifferential} of $f$ at $\boldsymbol x^*$ is the set $\partial f(\boldsymbol x^*)$ of vectors $\boldsymbol \xi \in \mathbb R^n$ such that
$$\partial f(\boldsymbol x^*) = \{\boldsymbol \xi \mid f^o(\boldsymbol x^*; \boldsymbol d) \geq \boldsymbol \xi^T \boldsymbol d,\ \forall\ \boldsymbol d \in \mathbb R^n\}.$$
Each $\boldsymbol \xi \in \partial f(\boldsymbol x^*)$ is called a \emph{subgradient} of $f$ at $\boldsymbol x^*$. 
\end{maar}

For a convex function, the above definition is equal to the subdifferential of a convex function \cite{Cla}. 
Clarke's subdifferential
has the so-called classical plus-minus symmetry, in other words 
\begin{equation}\label{symm}
\partial (-f(\boldsymbol x^*))=-\partial f(\boldsymbol x^*).
\end{equation}

In some books \cite{BMord, Mord} plus-minus symmetry (\ref{symm}) for Clarke's subdifferential of LLC functions was mentioned as a drawback leading to non-distinguishing between convex and concave functions, as well as local minima and maxima. However, the symmetry property might be crucial and desirable in many situations because it assures that Clarke's subdifferential behaves "nicely" under the negation of the function, maintaining its consistency and robustness as a generalization of the derivative concept to nonsmooth contexts.

A well-known result formulated and proven by Rademacher (see Theorem 9.60 in \cite{RW09}) states that an LLC function $f$ is differentiable almost everywhere. In particular, every neighborhood of $\boldsymbol x$ contains a point $\boldsymbol y$ such that $\nabla f(\boldsymbol y)$ exists. 
So, we have an equivalent definition for Clarke's subdifferential of LLC function 
\begin{equation}\label{Clarkesubdefeq}
\partial f(\boldsymbol x^*) = \textnormal{conv}\{\boldsymbol \xi \in \mathbb R^n\mid \exists\, \boldsymbol x^k\rightarrow \boldsymbol x^*,\ \nabla f(\boldsymbol x^k)\ \text{exists}\ \textnormal{and}\ \nabla f(\boldsymbol x^k)\rightarrow \boldsymbol \xi\},
\end{equation}
where $\textnormal{conv}$ denotes the $\emph{convex hull}$ of a set.
Let's notice that (\ref{Clarkesubdefeq}) can be conveniently used in practice to compute Clarke's subdifferential of an LLC function.

The \textit{subdifferential in convex analysis}, also called the \textit{convex subdifferential}, is defined for a \textit{convex function} \(f:\mathbb{R}^n \to \mathbb{R} \) at a point \( \boldsymbol{x^*} \) as:
\[
\partial_{conv} f(\boldsymbol{x^*}) = \{ \boldsymbol{\xi} \in \mathbb{R}^n : f(\boldsymbol{y}) \geq f(\boldsymbol {x^*}) + \boldsymbol{\xi}^\top (\boldsymbol{y} - \boldsymbol{x^*}), \, \forall\, \boldsymbol{y} \in \mathbb{R}^n \}.
\]
Naturally, any convex $f$ is also LLC \cite{CLLC}.

For a general LLC function \( f \), the Clarke subdifferential is typically larger than the convex subdifferential:
\[
\partial f(\mathbf{\boldsymbol x^*}) \supseteq \partial_{{conv}} f(\mathbf{\boldsymbol x^*}),
\]
with equality only when \( f \) is convex. \cite{Cla}

The next theorem shows that subgradients are generalizations of the classical gradient.
\begin{theorem} {\cite{BaKaMa}} If $f:\mathbb R^n \rightarrow \mathbb R$ is continuously differentiable at $\boldsymbol x \in \mathbb R^n$, then
$$\partial f(\boldsymbol x) = \{\nabla f(\boldsymbol x)\}.$$ \end{theorem}
\begin{proof} See, e.g. {\cite{BaKaMa}}, Theorem 3.7. \end{proof}

The Jacobian matrix of a vector-valued LLC function can be similarly generalized for the nonsmooth case.
\begin{maar}\cite{Cla}
The \emph{generalized Jacobian matrix} $J(\boldsymbol F((\boldsymbol{u})))$ of a nonsmooth vector-valued LLC function $\boldsymbol F = (f_1,\dots,f_p)$ at $\boldsymbol{u}$ is defined as  
$$J(\boldsymbol F(\boldsymbol{u})) = \big(\partial f_i(\boldsymbol{u})\big)_{i=1,\dots, p},$$
where $\partial f_i(\boldsymbol{u})$ is the Clarke subdifferential of $f_i$ at $\boldsymbol{u}$. In other words, the $i$:th row of the Jacobian is the Clarke subdifferential of $f_i$ at $\boldsymbol{u}$.
\end{maar}

\section{Single and vector-valued invex LLC functions}\label{sec3}

In this section, we discuss invex nonsmooth functions, which are generalizations of convex functions. Invex functions are rather large groups of functions retaining the useful property of convex functions, namely that, under mild conditions, we have the following (see Theorem 1 in \cite{Rei}):
\begin{itemize}
\item
for invex functions defined on $\mathbb R^n$ every stationary  point $\boldsymbol{u}\in \mathbb R^n$ (that is $\boldsymbol{0} \in \partial f(\boldsymbol{u})$) is a global minimum.     
\end{itemize}

First, we define invexity for single-valued nonsmooth functions and then generalize the concept of invexity for vector-valued nonsmooth functions.

Let us present the definition for an invex nonsmooth function, which was first conceived by Hanson {\cite{Hanson}} for differentiable functions. The following definition from {\cite{Rei}} generalizes the notion of invexity for nonsmooth functions.

\begin{maar}{\cite{Rei}} \label{invdef} An LLC function $f: \mathbb R^n \rightarrow \mathbb R$ is \emph{invex} if there exists a function $\boldsymbol \eta : \mathbb R^n \times \mathbb R^n \rightarrow \mathbb R^n$ such that for all $\boldsymbol x, \boldsymbol u\ \in\ \mathbb R^n$ and for all $\boldsymbol \xi \in \partial f(\boldsymbol u)$ we have
$$f(\boldsymbol x) - f(\boldsymbol u) \geq \boldsymbol \xi^T \boldsymbol \eta(\boldsymbol x; \boldsymbol u),$$
or equivalently
$$f(\boldsymbol x) - f(\boldsymbol u) \geq  f^o(\boldsymbol u; \boldsymbol \eta(\boldsymbol x; \boldsymbol u)).$$
\end{maar}

Here $f^o(\boldsymbol u; \boldsymbol \eta(\boldsymbol x; \boldsymbol u))$ is Clarke's generalized directional derivative defined at point $\boldsymbol u$ along the direction $\boldsymbol \eta(\boldsymbol x; \boldsymbol u)$ (see Definition \ref{GDD}).
By setting $\boldsymbol \eta(\boldsymbol x; \boldsymbol u) = \boldsymbol u - \boldsymbol x$, we can see that a convex function is a special case of an invex function.  

Hanson also gave a slightly relaxed definition for invexity in \cite{Hanson2}. These types of functions are called \emph{Type I invex functions}, and they form the most general class of functions for which the KKT conditions are necessary and sufficient for a global minimum. Unlike invex functions, invexity for Type I invex functions is defined relative to a fixed point $\boldsymbol{u}$, and for this reason, they are not necessarily invex for every point in $\mathbb R^n$.

The utility of invex functions was established when Hanson showed that if the objective and constraint functions of a nonlinear programming model are invex with respect to the same $\boldsymbol \eta$, then the weak duality and sufficiency of the KKT conditions still hold. The word 'invex' descends from a contraction of 'invariant convex', and Craven proposed it in \cite{Craven}. Craven and Glover \cite{CraGlo} as well as Ben-Israel and Mond \cite{BenMond} showed that the class of invex functions under the usual "closed cone" assumption is equivalent to the class of functions whose stationary points are global minima in unconstrained case. 

Notice that invexity can also be introduced for non-LLC functions. For example, in \cite{SiLa}, invexity is introduced for vector quasidifferentiable \cite{Dem} functions using quasidifferential instead of Clarke's subdifferential. More details on quasidifferential can be found in \cite{Craven2, Dem2, DemNiSha, DemPo}.

Let us now recall the classical definition of weakly continuous functions.

\begin{maar}\label{WC}\cite{RW09}
A function $f: \mathbb R^n \rightarrow \mathbb R$ is \emph{weakly continuous} if it is continuous with respect to the weak topology on \( \mathbb R^n \). This means that for any sequence \( \{\boldsymbol x_k\} \in \mathbb R^n \) that converges weakly to \(\boldsymbol x \), i.e.,  
\[
\ell(\boldsymbol x_k) \to \ell(\boldsymbol x) \quad \text{for all } \ell \in (\mathbb R^n)^*,
\]
it follows that
\[
f(\boldsymbol x_k) \to f(\boldsymbol x),
\]
where $\ell$ is a continuous linear functional, and  $(\mathbb R^n)^*$ is the dual space of $\mathbb R^n$.
\end{maar}

Intuitively, this means the sequence convergences when viewed through the "lens" of every bounded linear functional in the dual space. We have the following well-known result in finite-dimensional spaces $\mathbb R^n$ because the weak topology and the standard topology coincide.

\begin{lemma}\cite{RW09}\label{LLCWC}
Any LLC function $f:\mathbb R^n \rightarrow \mathbb R$ is also weakly continuous.    
\end{lemma} 

We will use this fact later when establishing some equivalence between Lemma \ref{altCra} and Lemma \ref{Thmalt}.

Now we introduce convexity and sublinearity with respect to an arbitrary closed convex cone.

\begin{maar}\label{Scon}\cite{CraGlo}
A function $f: \mathbb R^n \rightarrow \mathbb R$ is $S$-\emph{convex} if, whenever $0 < \alpha < 1$ and $\boldsymbol{x}, \boldsymbol{y} \in \mathbb R^n $, we have
$$\alpha f(\boldsymbol{x}) + (1 - \alpha) f(\boldsymbol{y}) - f(\alpha \boldsymbol{x} + (1 - \alpha)\boldsymbol{y}) \in S\subseteq \mathbb R,$$
where $S$ is a closed convex cone.
\end{maar}

\begin{maar}\label{Ssublin}\cite{CraGlo}
A function $f: \mathbb R^n \rightarrow \mathbb R$ is $S$-\emph{sublinear} if it is $S$-convex and positively homogeneous of degree one that is $f(\alpha \boldsymbol{x}) = \alpha f(\boldsymbol{x})\quad \forall\, \alpha\ge 0$.
\end{maar}

Note that when $S=\mathbb R_{\ge}$, $S$-convexity transforms into classical convexity as well as $S$-sublinearity transforms into usual sublinearity (subadditivity plus positive homogeneity). 

In order to formulate the next supplementary results assume that $\partial (\lambda f)=\lambda\partial f$ is a subdifferential of the function $\lambda f$, and $$\text{ray}\, S=\{\lambda:\ \lambda y\ge 0\ \ \forall\, y\in  S\}$$ is one-dimensional cone (ray)  dual to $S$. It is obvious that $\text{ray}\, S=\mathbb R_{\ge}$ if $S=\mathbb R_{\ge}$, and hence, $\lambda\ge 0$. 

The following lemma states two mutually exclusive properties for $S$-sublinear weakly continuous functions.

\begin{lemma}\label{altCra}\cite{CraGlo} (Craven\&Glover's lemma of alternative)
Let $f: \mathbb R^n \rightarrow \mathbb R$ be $S$-sublinear and weakly continuous, and let $z \in \mathbb R$. Then exactly one of the following is satisfied:
\begin{enumerate}
\item[(i)] $\exists\ \boldsymbol{x} \in \mathbb R^n\ \textnormal{such that}\ -f(\boldsymbol{x}) + z \in S$;
\item[(ii)] $(\boldsymbol{0}, -1)^T \in \mathrm{cl}\Big(\bigcup_{\lambda \in \text{ray}\, S} \big(\lambda\partial f(\boldsymbol{0}) \times \big\{\lambda z\big\}\big)\Big)$.
\end{enumerate}
\end{lemma}


The following known result is used as a characterization of invexity generalized for nonsmooth functions.
\begin{theorem}\label{3.1.2}{\cite{Rei}} Let $f:\mathbb R^n \rightarrow \mathbb R$ be LLC . For each $\boldsymbol u \in \mathbb R^n$, assume that for every $\boldsymbol x \in \mathbb R^n$, the convex cone in $\mathbb R^n\times\mathbb R$ (epigraphical tangent cone or a cone representation of the Clarke's subdifferential):

\begin{equation}\label{Contingent coneclosed}
  K_{\boldsymbol x}(\boldsymbol u) = \bigcup_{\lambda \geq 0} \Big(\lambda \partial f(\boldsymbol u) \times \big\{\lambda(f(\boldsymbol x) - f(\boldsymbol u))\big\}\Big)
\end{equation}
is closed. Then $f$ is invex if and only if every stationary point (subdifferential contains zero) is a global minimum of $f$ over $\mathbb R^n$.
\end{theorem}

Let us note that the assumption on closedness of $K_{\boldsymbol x}(\boldsymbol u)$ is not very restrictive. It only allows to exclude from consideration "trivial" cases since a nontrivial convex cone (one that doesn't cover the whole space) cannot be open. It will always contain its boundary (e.g., the apex or "edges" of the cone), making it at least closed or not open.

\begin{rem}
In the differentiable case, the convex cone defined in (\ref{Contingent coneclosed}) is always closed. For a differentiable single-valued function $f$,  we obtain famous result of \cite{CraGlo} that $f:\mathbb R^n\rightarrow \mathbb R$ is invex if and only if every stationary point (i.e. $\boldsymbol u$ such that $\nabla f(\boldsymbol u)=\boldsymbol 0$) is a global minimum of $f$ over $\mathbb R^n$.
\end{rem}

Now we are ready to adapt the Craven\&Glover's lemma of alternative (Lemma \ref{altCra}) to LLC functions and use it later in the proof of Theorem \ref{2.2.6}. 

\begin{lemma}\label{Thmalt}
Let $f: \mathbb R^n \rightarrow \mathbb R$ be LLC, $\boldsymbol x,\boldsymbol u \in \mathbb R^n$, and let $K_{\boldsymbol x}(\boldsymbol u)$ defined in (\ref{Contingent coneclosed}) be closed. Then exactly one of the following is satisfied:
\begin{enumerate}
\item[(i)] $\exists\ \boldsymbol{\eta}(\boldsymbol x;\boldsymbol u) \in \mathbb R^n$ such that $$f(\boldsymbol x) - f(\boldsymbol u) \geq \boldsymbol \xi^T \boldsymbol \eta(\boldsymbol x; \boldsymbol u),$$ or equivalently, $$f(\boldsymbol x) - f(\boldsymbol u) \geq  f^o(\boldsymbol u; \boldsymbol \eta(\boldsymbol x; \boldsymbol u));$$
\item[(ii)] $(\boldsymbol{0}, -1)^T \in K_{\boldsymbol x}(\boldsymbol u)$.
\end{enumerate}
\end{lemma}

\noindent
\begin{proof}
Taking into account Lemma \ref{LLCWC}, this result is obtained by setting $S=\mathbb R_{\ge}$ in Lemma \ref{altCra} and substituting $\boldsymbol{x}$ with $\boldsymbol \eta(\boldsymbol x;\boldsymbol u)$, $f(\boldsymbol{x})$ with $f^o(\boldsymbol{u};\boldsymbol \eta(\boldsymbol x;\boldsymbol u))$, and $z$ with $f(\boldsymbol{x})-f(\boldsymbol{u})$, respectively. 
The assumptions in Lemma \ref{altCra} on sublinearity (subadditivity and positive homogeneity) of function $f$ transform into assumptions of sublinearity of function $f^o(\boldsymbol{u}; \boldsymbol \eta(\boldsymbol x;\boldsymbol u))$ as a function of argument $\boldsymbol \eta(\boldsymbol x;\boldsymbol u)$, which always holds for all $\boldsymbol u\in \mathbb R^n$ in the case of Clarke's generalized directional derivative \cite{Cla}.
We also use a theorem from \cite{Cla} stating that for an LLC function $f$ at $\boldsymbol{u}$, we have $\partial f(\boldsymbol{u}) = \partial_{conv} f^o(\boldsymbol{u};\boldsymbol{0})$. 
\end{proof}

A careful reader may notice that condition (i) in Lemma \ref{Thmalt} will be equivalent to invexity of LLC function $f$ when fulfilled for all $\boldsymbol x,\boldsymbol u\in \mathbb R^n$.

Now we consider the following vector-valued function minimization problem
\begin{equation}\tag{VMP}
\min\limits_{ \boldsymbol x \in \mathbb R^n}\ \boldsymbol F(\boldsymbol x) = (f_1(\boldsymbol x),\dots,f_p(\boldsymbol x))^T\in\mathbb R^p.    
\end{equation}

Let us first define the most fundamental principle of optimality in vector optimization, known as Pareto optimality \cite{Pareto}.
\begin{maar}\label{Par}
A point $\boldsymbol x^* \in \mathbb R^n$ is said to be a \emph{Pareto optimal solution} of (VMP) if there exist no $\boldsymbol x \in \mathbb R^n$ such that
$$f_i(\boldsymbol x) \leq f_i(\boldsymbol x^*)\ \text{for all}\ i=1,\dots,p$$
and
$$f_i(\boldsymbol x) < f_i(\boldsymbol x^*)\ \textnormal{for some}\ i=1,\dots,p.$$
\end{maar}
We can modify the Pareto optimality principle by requiring strict inequality for all components in the vector dominance condition. This produces a milder optimality principle known as weak Pareto optimality. 

\begin{maar}\label{wPar}
A point $\boldsymbol x^* \in \mathbb R^n$ is said to be a \emph{weakly Pareto optimal solution} of (VMP) if there exist no $\boldsymbol x \in \mathbb R^n$ such that
$$f_i(\boldsymbol x) < f_i(\boldsymbol x^*)\ \text{for all}\ i=1,\dots,p.$$
\end{maar}
Therefore, it is easy to see that every Pareto optimal solution is also weakly Pareto optimal, but the opposite may not necessarily be true.

The concept of V-invexity was generalized for nonsmooth LLC functions in \cite{MisGio}.

\begin{maar}\label{vinvex} \cite{MisGio} Let $f_i : \mathbb R^n \rightarrow \mathbb R$ be LLC functions for all $i=1,\dots,p$. A vector-valued function $$\boldsymbol F(\boldsymbol x) = (f_1(\boldsymbol x),\dots,f_p(\boldsymbol x))^T$$ is \emph{V-invex} (with respect to $\boldsymbol \eta$) if there exist a function $\boldsymbol \eta : \mathbb R^n \times \mathbb R^n \rightarrow \mathbb R^n$ and a vector $\boldsymbol \beta$ with components $\beta_i : \mathbb R^n \times \mathbb R^n \rightarrow \mathbb R_{>0}$ such that for all $\boldsymbol x, \boldsymbol u \in \mathbb R^n$, for all $i=1,\dots,p$  and for all $\boldsymbol \xi^i \in \partial f_i(\boldsymbol u)$ we have
$$f_i(\boldsymbol x) - f_i(\boldsymbol u) \geq \beta_i(\boldsymbol x; \boldsymbol u) \boldsymbol (\boldsymbol \xi^i)^T \boldsymbol \eta(\boldsymbol x; \boldsymbol u),$$ or equivalently,
 $$f_i(\boldsymbol x) - f_i(\boldsymbol u) \geq  \beta_i(\boldsymbol x; \boldsymbol u)f_i^o(\boldsymbol u; \boldsymbol \eta(\boldsymbol x; \boldsymbol u)).$$

\end{maar}

Here we can write $\beta_i(\boldsymbol x; \boldsymbol u)\boldsymbol \eta(\boldsymbol x; \boldsymbol u) = \boldsymbol \eta_i(\boldsymbol x; \boldsymbol u)$, which shows that every $i$-th component function has a different $\boldsymbol \eta$-function, i.e. $\boldsymbol \eta_i$.  It is evident that a vector function whose components are invex with the same $\boldsymbol \eta$ is V-invex. To see that, it suffices to set $\beta_i \equiv 1$ for all $i=1,\dots,p$.

Note that Definition \ref{vinvex} was first introduced for differentiable functions in \cite{JeyMon}. It was shown by Jeyakumar and Mond \cite{JeyMon} that for a V-invex differentiable function $\boldsymbol F = (f_1,\dots,f_p)^T$ the point $\boldsymbol u \in \mathbb R^n$ is a weak Pareto optimal solution of (VMP) if and only if there exists $\boldsymbol \lambda \in \mathbb R^p_{\ge}\backslash \{\boldsymbol 0 \}$ such that 
\begin{equation}\label{taunabla}
\sum_{i=1}^p \lambda_i \nabla f_i(\boldsymbol u) = \boldsymbol 0.
\end{equation}

Let us extend (\ref{taunabla}) to nonsmooth functions. We will first need the following result (known as Gordan’s lemma of alternative).

\begin{lemma}\label{theoremGordan}\cite{Gordan} (Gordan's lemma of alternative A). For a given $m \times n$ matrix $A$  exactly one of the following two systems (but never both) has a solution: 
\begin{itemize}   
\item[(i)] primal system 
$A\boldsymbol{x}<(>)\,\boldsymbol{0}$ has a solution $\boldsymbol{x} \in \mathbb R^n$; 
\item[(ii)] dual system 
$A^T \boldsymbol{\lambda} = \boldsymbol{0},\ \boldsymbol{\lambda}\ge \boldsymbol{0},\ \boldsymbol{\lambda}\ne \boldsymbol{0}$ 
has a solution $\boldsymbol{\lambda} \in \mathbb R^m$.
\end{itemize}
\end{lemma}

We will use Lemma \ref{theoremGordan} to prove the following result, specifying the criteria for weak Pareto optimality for (VMP) with V-invex functions.

\begin{theorem}\label{thmVinv} Let $\boldsymbol F : \mathbb R^n \rightarrow \mathbb R^p$ be V-invex with $f_i : \mathbb R^n \rightarrow \mathbb R$ be LLC functions for all $i=1,\dots,p$. Given $\boldsymbol u \in \mathbb R^n$, assume the cone $K_{\boldsymbol{x}}(\boldsymbol u)$ defined in (\ref{Contingent coneclosed}) be closed for every $\boldsymbol x \in \mathbb R^n$.
Then $\boldsymbol u$ is a weakly Pareto optimal solution of (VMP)  associated with the vector-valued function $\boldsymbol F$  if and only if for all $\boldsymbol \xi:=(\boldsymbol\xi^1,\dots,\boldsymbol \xi^p),$ $\boldsymbol \xi^i \in \partial f_i(\boldsymbol u),\ i = 1, \dots, p$
 there exists $\boldsymbol \lambda \in \mathbb R^p_{\ge}\backslash \{\boldsymbol 0 \}$ such that
\begin{equation}\label{tauksi}\sum_{i=1}^p \lambda_i \boldsymbol \xi^i = \boldsymbol 0.
\end{equation}
\end{theorem}

\begin{proof} ($\Rightarrow$) Suppose that $\boldsymbol u$ is a weakly Pareto optimal solution of (VMP) associated with the vector-valued function $\boldsymbol F$. Then there exists no solution $\boldsymbol x\in \mathbb R^n$ such that $f_i(\boldsymbol x) < f_i(\boldsymbol u)$ for all $i=1,\dots, p$. In other words, the system $$\boldsymbol x \in \mathbb R^n,\ f_i(\boldsymbol x) < f_i(\boldsymbol u),\ i=1,\dots,p$$ is inconsistent. 
Since $\boldsymbol F$ is V-invex, there exist $\beta_i(\boldsymbol x; \boldsymbol u) > 0,\ i=1,\dots,p$ and $\boldsymbol \eta (\boldsymbol x; \boldsymbol u)$ such that for all $\boldsymbol \xi$, the system
\begin{equation}\label{etaxi}
\boldsymbol x \in \mathbb R^n,\ 0>\frac{1}{\beta_i(\boldsymbol x; \boldsymbol u)}(f_i(\boldsymbol x) - f_i(\boldsymbol u)) \geq (\boldsymbol{\xi}^i)^T \boldsymbol \eta(\boldsymbol x; \boldsymbol u),\ i=1,\dots,p
\end{equation}
is inconsistent. 
Then based on Gordan's lemma of alternative A we get that for all $\boldsymbol \xi$, 
there exists $\boldsymbol \lambda\in \mathbb R^p_{\ge}\backslash \{\boldsymbol 0 \}$ such that
$$\sum_{i=1}^p \lambda_i {\boldsymbol \xi}^i = \boldsymbol 0,
$$
in other words, (\ref{tauksi}) holds.

($\Leftarrow$) Assume that for all $\boldsymbol \xi$  there exists $\boldsymbol \lambda \in \mathbb R^p_{\ge}\backslash \{\boldsymbol 0 \}$ such that $\sum_{i=1}^p \lambda_i {\boldsymbol \xi^i} = \boldsymbol 0$.  Suppose that $\boldsymbol u$ is not a weakly Pareto optimal solution of (VMP). Then there exists $\boldsymbol x^* \in \mathbb R^n$ such that $f_i(\boldsymbol x^*) < f_i(\boldsymbol u),\ i=1,\dots,p$. Since $\boldsymbol F$ is V-invex, there exist $\beta_i(\boldsymbol x^*; \boldsymbol u) > 0,\ i=1,\dots,p$ and $\boldsymbol \eta (\boldsymbol x^*; \boldsymbol u)$ such that vector $\boldsymbol{x}^*$ is a solution for system (\ref{etaxi}) for all $\boldsymbol \xi$.
Since for all $\boldsymbol \xi$, we have  $\boldsymbol \lambda \in \mathbb R^p_{\ge}\backslash \{\boldsymbol 0 \}$, from the above one can get a contradiction with Gordan's lemma of alternative A.
\end{proof}

To formulate one more result, it is handy to use another version of Gordan's lemma of alternative given below.

\begin{lemma}\cite{Gordan} \label{Gordan's lemma of alternative B}(Gordan's lemma of alternative B). For a given $m \times n$ matrix $A$  exactly one of the following two systems (but never both) has a solution: 
\begin{itemize}
\item[(i)] primal system 
$A\boldsymbol{x}\le(\ge)\,\boldsymbol{0}$ has a solution $\boldsymbol{x} \in \mathbb R^n\backslash \{\boldsymbol 0\}$; 
\item[(ii)] dual system 
$A^T \boldsymbol{\lambda} = \boldsymbol{0}$ 
has a solution $\boldsymbol{\lambda} \in \mathbb R_{>}^m$.
\end{itemize}
\end{lemma}

The following result can be proven by analogy with the proof of Theorem \ref{thmVinv} and the help of Gordan’s lemma of alternative B (Lemma \ref{Gordan's lemma of alternative B}), 
It specifies the Pareto optimality criteria for (VMP) with V-invex functions.

\begin{theorem}\label{thmVinvpar} Let $\boldsymbol F : \mathbb R^n \rightarrow \mathbb R^p$ be V-invex with $f_i: \mathbb R^n \rightarrow \mathbb R$ be LLC functions for all $i=1,\dots,p$. Given $\boldsymbol u \in \mathbb R^n$, assume the cone $K_{\boldsymbol{x}}(\boldsymbol u)$ defined in (\ref{Contingent coneclosed}) be closed for every $\boldsymbol x \in \mathbb R^n$. Then $\boldsymbol u$ is a Pareto optimal solution of (VMP)  associated with the vector-valued function $\boldsymbol F$  if and only if for all $\boldsymbol \xi:=(\boldsymbol\xi^1,\dots,\boldsymbol \xi^p),$ $\boldsymbol \xi^i \in \partial f_i(\boldsymbol u),\ i = 1, \dots, p$
 there exists $\boldsymbol \lambda \in \mathbb R^p_{>0}$ such that
$$\sum_{i=1}^p \lambda_i \boldsymbol \xi^i = \boldsymbol 0.
$$
\end{theorem}

Note that Definition \ref{vinvex} transforms into the definition of invexity (see Definition \ref{invdef}) for the case $p=1$. Then the statement of Theorem \ref{thmVinvpar} transforms into the necessity statement of Theorem \ref{3.1.2}.

\section{Main result}\label{sec4}
In this section we generalize a result from \cite{Leg}  concerning the characterization families of differentiable functions that are invex with respect to a common function $\boldsymbol{\eta}$. For the completeness of the paper, we start with reproducing the result from \cite{Leg} along with the proof. First, we need the following known result.

\begin{lemma}\cite{Gale}\label{Gale's lemma of alternative for linear inequalities} (Gale's lemma of alternative). For a given $m \times n$ matrix $A$ and a given column vector $\boldsymbol{b} \in \mathbb R^m$, exactly one of the following two linear systems (but never both) has a solution: 
\begin{itemize}
\item[(i)]
primal system $A\boldsymbol{x}\leq \boldsymbol{b}$
has a solution $\boldsymbol{x} \in \mathbb R^n\backslash \{\boldsymbol 0\}$;
\item[(ii)]
dual system
$A^T \boldsymbol{\lambda} = \boldsymbol{0},\ \boldsymbol{b}^T \boldsymbol{\lambda} = -1$
has a solution $\boldsymbol \lambda\in \mathbb R_{\ge}^{m}$.
\end{itemize}
\end{lemma}

The next theorem characterizes invexity with respect to a common function $\boldsymbol{\eta}$ for differentiable functions.

\begin{theorem}\label{4.2}\cite{Leg}
Let $f_1,\dots,f_p$ be differentiable functions defined on $\mathbb R^n$. The following statements are equivalent:
\begin{itemize}
\item[(i)] The functions $f_1,\dots,f_p$ are invex with respect to the same $\boldsymbol{\eta}$.
\item[(ii)] The function $\sum_{i=1}^p \lambda_i f_i$ is invex with respect to the same $\boldsymbol{\eta}$ as the individual functions $f_i$, when $\lambda_i \geq 0,\ i=1,\dots,p$.
\item[(iii)] The function $\sum_{i=1}^p \lambda_i f_i$ is invex, when $\lambda_i \geq 0,\ i=1,\dots,p$.
\item[(iv)] For all $\lambda_i \geq 0,\ i=1,\dots,p$, every stationary point of $\sum_{i=1}^p \lambda_i f_i$ is a global minimum.
\end{itemize}
\end{theorem}


\begin{proof}
The implications (i) $\Rightarrow$ (ii) $\Rightarrow$ (iii) $\Rightarrow$ (iv) are obvious, so we only have to prove implication (iv) $\Rightarrow$ (i). Assume, by contradiction, that there is no function $\boldsymbol{\eta}$ such that $$f_i(\boldsymbol{x}) - f_i(\boldsymbol{u}) \geq \nabla f_i(\boldsymbol{u})^T\boldsymbol{\eta}(\boldsymbol{x};\boldsymbol{u}),\ \boldsymbol{x},\boldsymbol{u} \in \mathbb R^n,\ i = 1,\dots,p.$$ In other words, there exist $\boldsymbol{x},\boldsymbol{u} \in \mathbb R^n$ such that the linear inequality system
$$\nabla f_i(\boldsymbol{u})^T \boldsymbol{\eta}(\boldsymbol{x};\boldsymbol{u}) \leq f_i(\boldsymbol{x}) - f_i(\boldsymbol{u}),\ i=1,\dots,p$$
with the vector $\boldsymbol{\eta}(\boldsymbol{x};\boldsymbol{u})$ has no solution. Hence, by Gale's lemma of alternative (Lemma \ref{Gale's lemma of alternative for linear inequalities}), there exists $\boldsymbol \lambda:=(\lambda_1,\dots,\lambda_p)\in \mathbb R_{\ge}^{p}$ such that $\sum_{i=1}^p \lambda_i \nabla f_i(\boldsymbol{u}) = \boldsymbol{0}$ and $\sum_{i=1}^p \lambda_i(f_i(\boldsymbol{x}) - f_i(\boldsymbol{u})) = -1$. Therefore, $\sum_{i=1}^p \lambda_i f_i$ has a stationary point $\boldsymbol{u}$ which is not a global minimum, since $\sum_{i=1}^p \lambda_i f_i(\boldsymbol{x}) = \sum_{i=1}^p \lambda_i f_i(\boldsymbol{u}) - 1 < \sum_{i=1}^p \lambda_i f_i(\boldsymbol{u})$. This contradicts (iv).
\end{proof}

\begin{assum}\label{assumpt}
 In the nonsmooth case, we must assume 
the cone $K_{\boldsymbol{x}}$ defined in (\ref{Contingent coneclosed}) is closed, and we keep the assumption in the following considerations.
\end{assum}

\begin{rem}\label{subdifeq} In the nonsmooth case where for  functions $f_i:\mathbb R^n \rightarrow \mathbb R$, $i=1,\dots,p$, which are LLC at $\boldsymbol u\in \mathbb R^n$ and $\lambda_i\geq 0$ for all $i=1,\dots, p$, we have 
\begin{equation}\label{subdiffeqlab}
\partial(\sum_{i=1}^p \lambda_i f_i(\boldsymbol{u})) = \sum_{i=1}^p \lambda_i \partial f_i (\boldsymbol{u}),    
\end{equation} we can apply the proof in Theorem \ref{4.2} for vector-valued LLC functions in a quite straightforward way by replacing gradients with subgradients. 
\end{rem}
The equality condition (\ref{subdiffeqlab}) is known as the subdifferential weighted sum rule. It allows us to calculate the Clarke subdifferential of a weighted sum of functions as the weighted Minkowski sum of their individual Clarke's subdifferentials. As it was shown in \cite{Cla}, this condition is a consequence of Clarke's subdifferential regularity defined as follows.
\begin{maar}\label{regul}\cite{BaKaMa}
The LLC function $f:\mathbb R^n \rightarrow \mathbb R$ is called \emph{Clarke's subdifferentially  regular} at $\boldsymbol u\in \mathbb{R}^n$ if for all $\boldsymbol d\in \mathbb R^n$ the classical directional derivative $f'(\boldsymbol u; \boldsymbol d)$ exists and coincides with the generalized directional derivative $f^o(\boldsymbol u; \boldsymbol d)$, that is
$$f'(\boldsymbol u; \boldsymbol d)=f^o(\boldsymbol u; \boldsymbol d).$$
\end{maar}
The classical derivation rule for LLC functions is given as follows.
\begin{theorem}\cite{BaKaMa}
 Let $f_i:\mathbb R^n \rightarrow \mathbb R$, $i=1,\dots,p$, be LLC at $\boldsymbol u\in \mathbb R^n$ functions. Then
\begin{equation}\label{subdiffeqlabgen}
\partial(\sum_{i=1}^p \lambda_i f_i(\boldsymbol{u}))\subseteq \sum_{i=1}^p \lambda_i \partial f_i (\boldsymbol{u}),    
\end{equation} 
where $\lambda_i\in \mathbb R$ for all $i=1,\dots,p$.
In addition, if $f_i$ is subdifferentially regular in $\boldsymbol u$ and $\lambda_i\geq 0$ for all $i=1,\dots, p$, then equality occurs in (\ref{subdiffeqlabgen}), that is, (\ref{subdiffeqlab}) holds.
\end{theorem}

Thus, Clarke's subdifferential regularity allows us to move from the inclusion derivation rule (\ref{subdiffeqlabgen}) to the equation form of it, that is the subdifferential weighted sum rule (\ref{subdiffeqlab}).

\begin{rem}
Another case where subdifferential weighted sum rule (\ref{subdiffeqlab}) is valid is when at least $p-1$ of the functions $f_i$ are continuously differentiable (see, for example, \cite{Cla}). However, when there are two or more nonsmooth (LLC) functions, the implication (iv) $\Rightarrow$ (i)  in Theorem \ref{4.2} does not necessarily hold without Clarke's regularity assumption. Let us show this with the following example.
\end{rem}

\begin{example}\label{2.2.4}
Let us define LLC functions
$$f_1(x) = \begin{cases}
x\sin(\ln(x)) + x\cos(\ln(x)) + 5x, & x > 0 \\
6x, & x \leq 0,
\end{cases}$$
and $f_2(x) = 3|x|$. For $f_1$, we have $f_1^{'}(x) = 2\cos(\ln(x)) + 5,\ x > 0$. We have $f_1^{'}(x) \rightarrow [3,7]$, as $x \rightarrow 0_+$, and $f_1^{'}(x) = 6$, as $x \rightarrow 0_-$. Therefore, $\partial f_1(0) = [3,7]$. For $f_2$, we have $\partial f_2(0) = [-3,3]$. 

First, let us show that in our example (iv) is true for every $\lambda_i \geq 0$, $i=1,2$. We consider three cases: 
\begin{itemize}
\item {$\lambda_2 = 2\lambda_1:$}\\
In this case, we have $\sum_{i=1}^2\lambda_i f_i(x) = 0,\ x \le 0$ and $\sum_{i=1}^2\lambda_i f_i(x) > 0,\ x \geq 0$. So, all the stationary points $x\le 0$ are the global minima.

\item{$\lambda_2 > 2\lambda_1:$}\\
In this case, we have $\sum_{i=1}^2\lambda_i f_i(x) \ge 0$ when $x \le 0$. So, the only stationary point $0$ is the unique global minimum for $\sum_{i=1}^2\lambda_i f_i(x)$.
\item{$\lambda_2 < 2\lambda_1:$}\\
in this case
 $\sum_{i=1}^2\lambda_i f_i(x)$ has no stationary points. 
\end{itemize}
Thus, in all three cases, every stationary point of $\sum_{i=1}^2\lambda_i f_i(x)$ is a global minimum.

Now let us show that given (iv) is true, claim (i) does not hold. When we set 
$\lambda_1 = \lambda_2 = 1$, we have $0 \in \sum_{i=1}^2\lambda_i \partial f_i(0) = [0,10]$. However, for $\sum_{i=1}^2\lambda_i f_i(x)$, we have $\partial(\sum_{i=1}^p \lambda_i f_i(0)) = [3,10]$, making $u=0$ non-stationary.
Let us take 
$x=-\frac{1}{3}$. Then we have $\sum_{i=1}^2\lambda_i (f_i(x) - f_i(0)) = 1(-2+0) +1 (1+0) = -1$. We can see that $$(0,-1)^T \in (\sum_{i=1}^2\lambda_i \partial f_i(0), \sum_{i=1}^2\lambda_i (f_i(x) - f_i(0))).$$ Then by lemma of alternative (Lemma \ref{Thmalt}), there exists no $\eta(x;u)$ for which $f_i(x)-f_i(u) \geq f_i^o(u,\eta(x;u))$ for both $f_1$ and $f_2$, when $x=-\frac{1}{3}$ and $u=0$. 

Indeed, for $x=-\frac{1}{3}$ and $u = 0$, we have $f_1(x) - f_1(u) = -2$. For the condition $f_1(x) - f_1(u) \geq f_1^o(u; \eta(x;u))$ to hold, we need to have $\eta(x;u) \leq -\frac{2}{3}$. For $f_2$, we have $f_2(x) - f_2(u) = 1$. Since we need to have $\eta(x;u) \leq -\frac{2}{3}$, we have $f_2^o(u; \eta(x;u)) \geq -3(-\frac{2}{3}) = 2$. Then we can see that the condition $f_2(x) - f_2(u) \geq f_2^o(u; \eta(x;u))$ does not hold for $x=-\frac{1}{3}$ and $u = 0$. Therefore, functions $f_1$ and $f_2$ are not invex for the same function $\eta$. Thus, our example has no implication (iv) $\Rightarrow$ (i).

\end{example}

\begin{rem}
Theorem \ref{4.2} can also be formulated for the nonsmooth case without the regularity assumptions given in Remark \ref{subdifeq} by rewriting (iv) as follows: Given Assumption \ref{assumpt} holds, for every $\lambda_i \geq 0,\ i=1,\dots,p$, 
\begin{equation}\label{l1}
\sum_{i=1}^p \lambda_i \boldsymbol{\xi}^i=\textbf{0},\ \boldsymbol{\xi}^i \in \partial f_i(\boldsymbol{u})
\end{equation}
implies that $\boldsymbol{u}$ is a global minimum of $\sum_{i=1}^p \lambda_i f_i$. Indeed, since $\boldsymbol{u}$ is a global minimum, we have 
\begin{equation}\label{l2}\textbf{0} \in \partial(\sum_{i=1}^p \lambda_i f_i(\boldsymbol{u})),
\end{equation}
meaning that \begin{equation}\label{l3}
\textbf{0} \in \sum_{i=1}^p \lambda_i \partial f_i (\boldsymbol{u})\implies\textbf{0} \in \partial(\sum_{i=1}^p \lambda_i f_i(\boldsymbol{u})).
\end{equation} 
\end{rem}

The theorem can be generalized to V-invex functions (see Definition \ref{vinvex}). The proof is a generalization of the proofs used in Theorem 3 in \cite{Rei} (see Theorem \ref{3.1.2}) and Theorem \ref{4.2}.

\begin{theorem}\label{2.2.6}\label{MTh}
Let $f_1,\dots, f_p$ be LLC functions defined on $\mathbb R^n$. Let $J(\boldsymbol F(\boldsymbol{u}))$ be the generalized Jacobian of $\boldsymbol F=(f_1, \dots, f_p)^T$ at $\boldsymbol{u}$ and 
$$\boldsymbol{b}:=\boldsymbol{b}(\boldsymbol{x};\boldsymbol{u}) = \big(\beta_1(\boldsymbol{x};\boldsymbol{u})(f_1(\boldsymbol{x}) - f_1(\boldsymbol{u})), \dots,\ \beta_p(\boldsymbol{x};\boldsymbol{u})(f_p(\boldsymbol{x}) - f_p(\boldsymbol{u}))\big)^T,$$ 
where $\boldsymbol \beta=(\beta_1,\dots\beta_p)^T$ is a vector of the coefficient functions $\beta_i:\mathbb R^n \times \mathbb R^n \rightarrow \mathbb R_{>0}$ from Definition \ref{vinvex}. For each $\boldsymbol u \in \mathbb R^n$, assume that for every $\boldsymbol x \in \mathbb R^n$, the convex cone
$$
K_{\boldsymbol x}(\boldsymbol u) = \bigcup_{\boldsymbol \lambda \geq 0} \Big(J(\boldsymbol F(\boldsymbol{u}))^T \boldsymbol \lambda \times \{\boldsymbol \lambda^T \boldsymbol {b}(\boldsymbol{x};\boldsymbol{u})\}\Big)
$$
is closed. The following statements are equivalent:
\begin{itemize}
\item[(i)] 
For a specified function coefficient  vector $\boldsymbol \beta\in \mathbb R^p_{>0}$, the function $\boldsymbol F : \mathbb R^n \rightarrow \mathbb R^p$ is V-invex, that is the functions $f_1,\dots,f_p$ are invex with respect to the same $\boldsymbol{\eta}$. 
\item[(ii)] 
For a specified function coefficient  vector $\boldsymbol \beta\in \mathbb R^p_{>0}$, the scalarized function 
$\sum_{i=1}^p \lambda_i \beta_if_i$ with $\lambda_i \geq 0,\ i=1,\dots,p$, is invex with respect to the same $\boldsymbol{\eta}$, that is
$$\sum_{i=1}^p \lambda_i \beta_i(\boldsymbol{x};\boldsymbol{u})((f_i(\boldsymbol{x}) - f_i(\boldsymbol{u})) \geq \sum_{i=1}^p \lambda_i (\boldsymbol{\xi}^i)^T \boldsymbol{\eta}(\boldsymbol x;\boldsymbol u)$$ 
$$\forall \lambda_i \geq 0,\ \boldsymbol{x}, \boldsymbol{u} \in \mathbb R^n,\ \boldsymbol{\xi}^i \in \partial f_i(\boldsymbol{u}).$$
\item[(iii)] For a specified function coefficient  vector $\boldsymbol \beta\in \mathbb R^p_{>0}$, the scalarized function $\sum_{i=1}^p \lambda_i \beta_if_i$ is invex, when $\lambda_i \geq 0,\ i=1,\dots,p$.
\item[(iv)] For every $\lambda_i \geq 0,\ i=1,\dots,p$, and 
for a specified function coefficient  vector $\boldsymbol \beta\in \mathbb R^p_{>0}$, the fact that 
$$ \sum_{i=1}^p \lambda_i \boldsymbol{\xi}^i=\textbf{0}$$
for some $\boldsymbol{\xi}^i \in \partial f_i(\boldsymbol{u})$ implies that $\boldsymbol{u}$ is a global minimum of the scalarized function 
\begin{equation}\label{scal}
    \sum_{i=1}^p \lambda_i \beta_if_i.
\end{equation}
\end{itemize}
\end{theorem}

\begin{proof}
The implications (i) $\Rightarrow$ (ii) $\Rightarrow$ (iii) $\Rightarrow$ (iv) are obvious, so it suffices to prove (iv) $\Rightarrow$ (i). From (iv) we can see that $(\boldsymbol{0},-1)^T \notin K_{\boldsymbol{x}}(\boldsymbol u)$, because otherwise we would have $\sum_{i=1}^p \lambda_i \boldsymbol{\xi}^i=\textbf{0}$ and $\sum_{i=1}^p \lambda_i \beta_i(\boldsymbol{x};\boldsymbol{u})f_i(\boldsymbol{x}) < \sum_{i=1}^p \lambda_i \beta_i(\boldsymbol{x};\boldsymbol{u})f_i(\boldsymbol{x}) + 1 = \sum_{i=1}^p \lambda_i \beta_i(\boldsymbol{x};\boldsymbol{u})f_i(\boldsymbol{u})$, which is a contradiction to the implication that $\boldsymbol{u}$ is a global minimum of the scalarized function (\ref{scal}). Therefore, by the lemma of alternative (Lemma \ref{Thmalt}) applied (considering $K_{\boldsymbol x}(\boldsymbol u)$ is closed and may generally speaking be different for each $\boldsymbol u$) to the vector-valued  function $\boldsymbol{\beta}\boldsymbol{F}$ (each $f_i$ is LLC, $\beta_i\in \mathbb R_{>0}$), there exists $\boldsymbol \eta(\boldsymbol x; \boldsymbol u)$ such that $$\boldsymbol b_i=\beta_i(\boldsymbol{x};\boldsymbol{u})(f_i(\boldsymbol x) - f_i(\boldsymbol u))\geq   \beta_i(\boldsymbol{x};\boldsymbol{u}){f}^o_i(\boldsymbol u; \boldsymbol \eta(\boldsymbol x; \boldsymbol u)).$$ Thus, $\boldsymbol F : \mathbb R^n \rightarrow \mathbb R^p$ is V-invex.
\end{proof}





\section{Conclusion}\label{sec5}
We generalized the result of \cite{Leg} for nonsmooth LLC invex single- and vector-valued functions. Straightforward generalization was not possible and it was shown by counterexample. However, a small modification of one of the conditions did the trick.
In the future, we plan to consider if similar results can be obtained with other nonsmooth methodologies beyond the conventional Clarke's subdifferential. For example, one of the possibilities is to follow \cite{SiLa}, where invexity was introduced for vector quasidifferentiable \cite{Dem} functions using quasidifferential instead of Clarke's subdifferential.


\begin{thebibliography}{100}




\bibitem{BaKaMa} A.~Bagirov, N.~Karmitsa, M.~M.~Mäkelä: \emph{Introduction to Nonsmooth Optimization: Theory, Practice and Software}, Springer, Cham, Switzerland, 2014.

\bibitem{BenMond} A.~Ben-Israel, B.~Mond: \emph{What is invexity?}, J. Austral. Math. Soc. Ser. B, vol. 28, 1986, pp. 1-9.

\bibitem{Cla} F.~Clarke: \emph{Optimization and Nonsmooth Analysis}, Wiley, New York, 1983.

\bibitem{CravenApp}
B.~Craven: \emph{Invexity and its applications}. In: Floudas, C., Pardalos, P. (eds) Encyclopedia of Optimization. Springer, Boston, 2008, pp. 1770–1774. 

\bibitem{Craven} B.~Craven: \emph{Invex functions and constrained local minima}, Bull. Austral. Math. Soc., vol. 24, 1981, pp. 357-366.

\bibitem{Craven2} B.~Craven: \emph{On quasidifferentiable optimization}, J. Aust. Math. Soc., vol. 41, 1986, pp. 64-78.


\bibitem{CraGlo} B.~Craven, B.~Glover: \emph{Invex functions and duality}, J. Austral. Math. Soc. Ser. A., vol. 39, 1985, pp. 1-20.


\bibitem{Dem}
V.~F.~Demyanov, A.~M.~Rubinov: \emph{On quasidifferentiable functional}, Proceed. of the USSR Acad. of Sci., vol. 21, 1980, pp. 14–17.

\bibitem{Dem2} V.~F.~Demyanov: \emph{Quasidifferentiable functions: necessary conditions and descent directions}, Math. Program. Stud., vol. 29, 1986, pp. 20-43.

\bibitem{DemNiSha} V.~F.~Demyanov, V.~N.~Nikulina, I.~R.~Shablinskaya: \emph{Quasidifferentiable functions in optimal control}, Math. Program. Stud., vol. 29, 1986, pp. 160-175.

\bibitem{DemPo} V.~F.~Demyanov, L.~N.~Polyakova: \emph{Minimization of a quasidifferentiable function on a quasidifferentiable set}, Comput. Math. and Math. Phys, vol. 20, 1980, pp. 34-43.

\bibitem{Gale} D.~Gale: \emph{The Theory of Linear Economic Models}, McGraq-Hill, New York, 1960.

\bibitem{Gordan}  P.~Gordan: \emph{Über die auflösung linearer gleichungen mit reellen coefficienten}, Math. Ann., vol. 6 (1), 1873, pp. 23–28.

\bibitem{Hanson} M.~Hanson: \emph{On sufficiency of the Kuhn-Tucker conditions}, J. Math. Anal. Appl., vol. 80, 1981, pp. 544-550.

\bibitem{Hanson2} M.~Hanson: \emph{Invexity and the Kuhn-Tucker theorem},  Math. Anal. Appl., vol. 236 (2), 1999, pp. 594-604.

\bibitem{JeyMon} V.~Jeyakumar, B.~Mond: \emph{On generalised convex mathematical programming}, J. Austral. Math. Soc. Ser. B,  vol. 34, 1992, pp. 43-53.

\bibitem{Martin} D.~Martin \emph{The essence of invexity}, J. Optim. Theor. Appl., vol. 47, 1985, pp. 65-76.

\bibitem{Leg} J.~E.~Martinez-Legaz: \emph{What is invexity with respect to the same $\boldsymbol{\eta}$?}, Taiwan. J. Math., vol. 13 (2B), 2009,  pp. 753-755.


\bibitem{MisGio} S.~Mishra, G.~Giorgi: \emph{Invexity and Optimization}, Nonconvex Optimization and Its Applications, vol. 88, Springer-Verlag Berlin Heidelberg, 2008.

\bibitem{BMord} B.~Mordukhovich: \emph{Variational Analysis and Applications}, Springer, Cham, Switzerland, 2018.

\bibitem{Mord} B.~Mordukhovich, M.~Nguyen: \emph{Convex Analysis and Beyond: Volume I: Basic Theory}, Springer Series in Operations Research and Financial Engineering, 1st ed., 2022.

\bibitem{Pareto} V.~Pareto: \emph{The Mind and Society [Trattato Di Sociologia Generale]}, Harcourt, 1935.

\bibitem{Rei} T.~Reiland: \emph{Nonsmooth invexity}, Bull. Austral. Math. Soc., vol. 42, 1990, pp. 437-446.

\bibitem{RW09}
R.~T.~Rockafellar, T.~J-B. Wets. \emph{Variational Analysis}, Grundlehren der mathematischen Wissenchaften, vol. 317. Springer Science \& Business Media, Berlin, 2009.

\bibitem{SchaZi} S.~Schaible, W~Ziemba: \emph{Duality for generalised convex fractional programs}, In: Generalised concavity in Optimisation and Economics, Academic Press, New York, 1981, pp. 473-490.



\bibitem{SiLa}
H.~N.~Singh, V.~Laha, \emph{On Minty variational principle for quasidifferentiable vector optimization problems}, Optim. Methods Softw., vol. 38, 2023, pp. 243–261. 


\bibitem{CLLC} Wayne State University Coffee Room, \emph{Every convex function is locally Lipschitz}, The Amer. Math. Month., vol. 79, 1972, pp. 1121-1124.

\end{thebibliography}
\end{document}